\documentclass[12pt, letterpaper,reqno]{amsart}
\usepackage[utf8]{inputenc}
\usepackage{amsmath}
\usepackage{amssymb}
\usepackage{datetime}
\usepackage{fancyhdr}
\usepackage{graphicx}
\usepackage{latexsym}
\usepackage{pinlabel}
\usepackage{fullpage}
\usepackage{enumitem}
\usepackage[unicode]{hyperref}
\usepackage{todonotes}

\usepackage{float}

\usepackage{listings}
\usepackage{xcolor}

\hypersetup{
    unicode=false,          
    pdftoolbar=true,        
    pdfmenubar=true,        
    pdffitwindow=false,     
    pdfstartview={FitH},    
    pdftitle={},    
    pdfauthor={Author},     
    pdfnewwindow=true,      
    colorlinks=true,       
    linkcolor=red,          
    citecolor=blue,        
    filecolor=magenta,      
    urlcolor=cyan,           
}

\newtheorem{theorem}{Theorem}[section]

\newtheorem{remark}[theorem]{Remark}
\newtheorem{example}[theorem]{Example}

\newcommand{\C}{{\mathbb C}}

\newcommand{\bit}{\begin{itemize}}
\newcommand{\eit}{\end{itemize}}

\newcommand{\ben}{\begin{enumerate}}
\DeclareRobustCommand{\een}{ \end{enumerate} }

\title{Construction and Periodicity of Minimal Resolution Graphs of Suspension Singularities}
\author[S.\ Sakall{\i}]{S\"umeyra Sakall{\i}}
\email{sumeyras@usf.edu}\urladdr{https://sites.google.com/view/smyrsakalli/home}
\address{Department of Mathematics and Statistics, University of South Florida, Tampa, FL, 33620, USA}
\author[C. \ Salles]{Claudia Salles}
\email{csallesgallo@usf.edu}
\address{University of South Florida, Tampa, FL, 33620, USA}

\begin{document}
\begin{abstract} 
We develop computational tools for constructing the minimal resolution graphs of suspension singularities. In particular, we provide a Magma program that implements Némethi’s algorithm that can be used to construct the minimal resolution graph of any suspension singularity. The program outputs the resulting resolution graph as structured graph data, encoding the vertices, edges, and decorations of the graph. We further develop a SageMath program that takes this data as input and produces graphical representations of the resolution graphs, facilitating their visualization and analysis. For three families of suspension singularities, we prove that the corresponding minimal resolution graphs are periodic, with periods equal to the orders of the corresponding monodromy factorizations. We also provide Magma programs for generating compactified resolution graphs of suspension singularities and checking graph isomorphisms over a specified range of parameters.
\end{abstract}

\keywords{Complex singularities, Symplectic 4-manifolds, Némethi’s algorithm, Monodromy factorizations}
\subjclass{57K43, 53D05, 32S30, 14J17}

\maketitle
\section{Introduction}
Singular fibers in holomorphic genus two fibrations were completely classified by Namikawa--Ueno \cite{NamikawaUeno-list, NamikawaUeno-long}. In \cite{SV} and \cite{SV2} we worked with Namikawa--Ueno's fibers with finite order monodromies which are analogous to Kodaira's fibers \cite{Kodaira} in holomorphic genus one fibrations. 

Recall that for $f: (\C^2,0) \rightarrow (\C,0)$ an isolated plane curve singularity, hypersurface singularities of the form $(\{f(x,y)+t^k=0\},0) \subset (\C^3,0)$ are called suspension singularities. In (\cite{Nemethi-lectures}, III. Appendix 1) Némethi provides an algorithm for determining the minimal resolution graphs of suspension singularities. Recall also that a minimal resolution graph means that the graph contains no vertices corresponding to $(-1)$-spheres.

We observed in \cite{SV} that the Namikawa--Ueno fibers under consideration are of suspension type. We reconstructed most of these fibers by applying Némethi's algorithm to their defining polynomials. In \cite{SV2}, we reconstructed the remaining four types by determining their defining polynomials different from those proposed by Namikawa and Ueno. In each of these papers, we then carefully deformed the fibrations into Lefschetz fibrations and computed their monodromy factorizations, providing effective tools for constructing new exotic 4-manifolds. 

In constructing the Namikawa-Ueno fibers, we manually applied Némethi's algorithm to all 37 members of the families introduced in \cite{SV} and \cite{SV2}: 

\begin{equation} \label{fam}
\begin{aligned}
\phi_1^k &: y^2 = x^5+t^k, \quad k=1,\dots, 10, \\
\phi_2^k &: y^2 = x(x^4+t^k), \quad k=1,\dots, 8,  \\
\phi_3^k &: y^2 = x^6+t^k, \quad k=1,\dots, 6,  \\
\phi_4^k &: y^2 = x(x^5+t^k), \quad k=1,\dots, 5,\\
\psi_1^k &: t^k = x(x^3-y^2), \quad k=1,\dots, 8.
\end{aligned}
\end{equation}
To the best of our knowledge, no existing computer program implementing Némethi's algorithm is available in the literature. Although these computations are straightforward in principle, carrying them out by hand is a highly time-consuming process (see also, for example, the related discussion in \cite{MathOverflow}). Moreover, a manual approach is not feasible for computations involving higher-degree polynomials.

This motivates the development of computational tools that automate Némethi's algorithm and make these calculations feasible for arbitrary suspension singularities, including those defined by high-degree polynomials. In this paper, we develop these tools, providing an efficient and reproducible framework for computations that would otherwise require extensive calculations by hand. In \autoref{Magma1} we present a Magma program \verb|PrintGraphwithArrows.m| for computing the minimal resolution graphs of suspension singularities. This program implements Némethi's algorithm and outputs the resulting resolution graphs as graph data, consisting of lists and dictionaries that encode the vertices, edges, and decorations of the
graph. We then provide a SageMath program \verb|PlotGraph (Sage)| that takes the resulting data as input and produces graphical representations of the corresponding resolution graphs, allowing for their effective visualization and analysis. The programs \verb|PrintGraphwithArrows.m| and \verb|PlotGraph (Sage)| can be used to construct the minimal resolution graph of \emph{any} suspension singularity. In \autoref{Theorem}, we prove that for each of the following infinite families of suspension singularities:
\begin{equation} 
\begin{aligned}
\phi_1^k &: y^2 = x^5+t^k, \quad k \in \mathbb{Z}_{>0}, \\
\phi_3^k &: y^2 = x^6+t^k, \quad k \in \mathbb{Z}_{>0},  \\
\psi_1^k &: t^k = x(x^3-y^2), \quad k \in \mathbb{Z}_{>0},
\end{aligned}
\end{equation}
the corresponding compactified minimal resolution graphs are periodic, with the period equal to the order of the corresponding monodromy factorization (see \autoref{T}). These minimal graphs are precisely the fibers in the Namikawa–Ueno classification having finite-order monodromies. Consequently, they exhaust all genus two singular fibers with finite-order monodromies. In \autoref{Magma2}, we provide a Magma program \verb|PrintGraphCompactified.m| that generates the compactified resolution graphs of suspension singularities. In \autoref{Isom}, we provide another Magma program \verb|IsomorphismCheck.m| that checks whether graphs corresponding to the same congruence class modulo a proposed period are isomorphic over a specified range of parameters. These programs facilitate efficient computation and comparison of compactified resolution graphs, providing a framework to explore further examples.

\section{Magma and SageMath Programs}
\label{Magma1}

We first briefly review Némethi's algorithm described in \cite{Nemethi-lectures}, III, Appendix 1. To determine the minimal resolution graph of a suspension singularity $S= (\{f(x,y)+t^k=0\},0)$ in $\mathbb{C}^3$, one first resolves the plane curve singularity $B= (\{f(x,y)=0\},0)$ in $\mathbb{C}^2$ by successive blow ups. This process is well known (see, for instance, \cite{GS}, Section 7.2; alternatively, the well known Magma function \verb|ResolutionGraph| can be used to compute the resolution graph of the plane curve singularity $B$ in the form of a matrix). In the resulting resolution graph $G$, which is a tree, each arrow denotes the proper transform of an irreducible component of the singular curve $B$. Each vertex represents a rational curve that is an exceptional sphere or its proper transform, and an edge between two vertices denotes the intersection of the corresponding curves. 

To find the resolution graph of the suspension singularity $S$, Némethi's algorithm prescribes how to take a special degree-$k$ cover of the graph $G$. Namely, one first determines the preimages of every vertex, arrowhead, and edge of $G$. The preimages of edges are called \emph{strings} and generally contain additional internal vertices. For each vertex of the covering graph, one then computes its self-intersection, multiplicity, and genus.  
Finally, one drops all arrows, and then blows down every exceptional sphere with self-intersection $(-1)$ to obtain the minimal resolution graph of $S$.

The workflow is summarized in the following schematic diagram: 
\[
\boxed{
\begin{aligned}
&\text{Resolve plane curve singularity } f(x,y)
\longrightarrow \text{obtain its resolution graph } G \longrightarrow\\
& \text{apply Némethi's algorithm to } G
\longrightarrow \text{obtain the minimal resolution graph of } f(x,y)+t^k.
\end{aligned}}
\]

\begin{remark}
Although in Némethi's algorithm all arrows are dropped before the final blowdown step, they play an important role in the computation of the corresponding monodromy factorizations and in the constructions of exotic 4-manifolds \cite{SV, SV2}. Therefore, throughout this paper we retain all arrows in the resulting resolution graphs.
\end{remark}

While there is the Magma function \verb|ResolutionGraph| that can be used to compute the resolution graphs of plane curve singularities, we are not aware of an online tool that implements Némethi's algorithm and that computes the resolution graphs of suspension singularities. We now provide two programs implementing Némethi's algorithm: the Magma program \verb|PrintGraphwithArrows.m| for computing the minimal resolution graphs of suspension singularities with the arrows retained, and the SageMath program \verb|PlotGraph (Sage)| for visualizing them. The programs are available at \cite{SSG}. 

The Magma program \verb|PrintGraphwithArrows.m| takes as input a suspension singularity of the form \(t^k=f(x,y)\) and outputs the data of its resolution graph, without requiring the resolution graph of the plane curve singularity $f(x,y)=0$ to be computed separately. To run the program, one first downloads the file \verb|PrintGraphwithArrows.m| from \cite{SSG}. The program contains the following input section:
\begin{lstlisting}
A<x,y> := AffineSpace(Rationals(),2);

// ______________input________________

//Input the value for the power of t:
k := ;
//Input a polynomial in terms of x and y (and k if needed):
curve := ;
\end{lstlisting}
Here, \verb|k|, the exponent of \(t\), is the degree of the cover in Némethi's algorithm, and \verb|curve| specifies the defining polynomial \(f(x,y)\) of the plane curve singularity. To compute the minimal resolution graph of \(t^k=f(x,y)\), one enters the appropriate values of \verb|k| and \verb|curve| in the input section. 

After starting Magma on a computer, the downloaded program is loaded by specifying its location (replacing \verb|/path/to/| with the directory where the file is stored):

\begin{lstlisting}
load "/path/to/PrintGraphwithArrows.m";
\end{lstlisting}


The program can also be run using the online Magma calculator available at 
\url{https://magma.maths.usyd.edu.au/calc/} 
by pasting the contents of \verb|PrintGraphwithArrows.m| into the calculator. We found that the online calculator is capable of handling polynomials of remarkably high degree.

At the end of its output, \verb|PrintGraphwithArrows.m| prints the graph data in a format that can be copied directly into SageMath for visualization. 

To visualize the resulting resolution graph, one downloads \verb|PlotGraph (Sage)| from \cite{SSG}. The end portion of the Magma output that begins immediately after the line
\begin{lstlisting}
Copy the following into SageMath to produce the REDUCED GRAPH
\end{lstlisting}
is copied into \verb|PlotGraph (Sage)| above the line
\begin{lstlisting}
Paste the lines printed by Magma above this line
\end{lstlisting}

Finally, one opens \url{https://sagecell.sagemath.org}, pastes the contents of \verb|PlotGraph (Sage)| into the SageMath cell, and evaluates the code to produce a graphical representation of the resolution graph. In the resulting graph, green disks denote arrows, blue disks denote sphere components, and blue squares denote genus $g$ surfaces with $g>0$. On each component of the plotted graph, the number in brackets denotes its multiplicity, the number in parentheses denotes its self-intersection, and $g$ denotes its genus. (The remaining number is a vertex label used only for identification). For instance, a label of the form $[3]$, $(-2)$, $g:1$ indicates a genus one component with multiplicity $3$ and self-intersection $(-2)$.

\begin{remark} \label{bdown}
As part of the resolution process, the Magma program \verb|PrintGraphwithArrows.m| blows down all $(-1)$-vertices of degree one or two. The only exception in degree two occurs when a $(-1)$-sphere appears at the end of a multigraph and is connected to the same vertex by two parallel edges. In this case, the program does not blow down the $(-1)$-sphere, since after the blowdown we obtain a nodal singularity. Thus, this blowdown must be performed manually. Any remaining $(-1)$-vertices of degree greater than two, if present, must also be blown down manually to obtain the minimal resolution graph. We also note that N\'emethi's algorithm constructs special degree $k$ covers of trees. Consequently, the resulting resolution graphs are trees or multigraphs, but never contain loops. Therefore, the blowdown procedure in \verb|PrintGraphwithArrows.m| assumes that the input graph has no loops.
\end{remark}

We now present applications of our programs to several examples from \cite{SV, SV2}, N\'emethi's lecture notes \cite{Nemethi-lectures}, and also \cite{GS}.

\begin{example}
To construct the minimal resolution graph of the suspension singularity \(t^3=-x^5+y^2\), which is denoted by \(\phi_1^3\) in \cite{SV}, one enters:
{
\normalfont
\begin{lstlisting}
k := 3;
curve := -x^5+y^2;
\end{lstlisting}}

\noindent into the input section of \verb|PrintGraphwithArrows.m|. The final portion of the output is:
{
\normalfont
\begin{lstlisting}
 Copy the following into SageMath to produce the REDUCED GRAPH
vertex_mult = {0: 10, 1: 4, 2: 1, 3: 5, 4: 8, 5: 6, 6: 7, 7: 4, 8: 2}
vertex_selfi = {0: -2, 1: -2, 2: -4, 3: -2, 4: -2, 5: -2, 6: -2, 7: -2, 8: -2}
vertex_genus = {0: 0, 1: 0, 2: 0, 3: 0, 4: 0, 5: 0, 6: 0, 7: 0, 8: 0}
edge_list = [(0,3), (0,4), (0,6), (1,5), (1,8), (2,7), (4,5), (6,7)]
arrow_vertices = [2]
\end{lstlisting}}

The portion of the output above, following the line 

\verb|Copy the following into SageMath to produce the REDUCED GRAPH| 

is then pasted into the SageMath program \verb|PlotGraph (Sage)|. Evaluating it in the online SageMath cell produces the graph shown in Figure \ref{fig:phi13}, which is the minimal resolution graph of \(\phi_1^3\) as computed in \cite{SV} (see the first row, second column of Figure~4 in \cite{SV}).
\begin{figure}
\begin{center}
\includegraphics[width=0.7\linewidth]{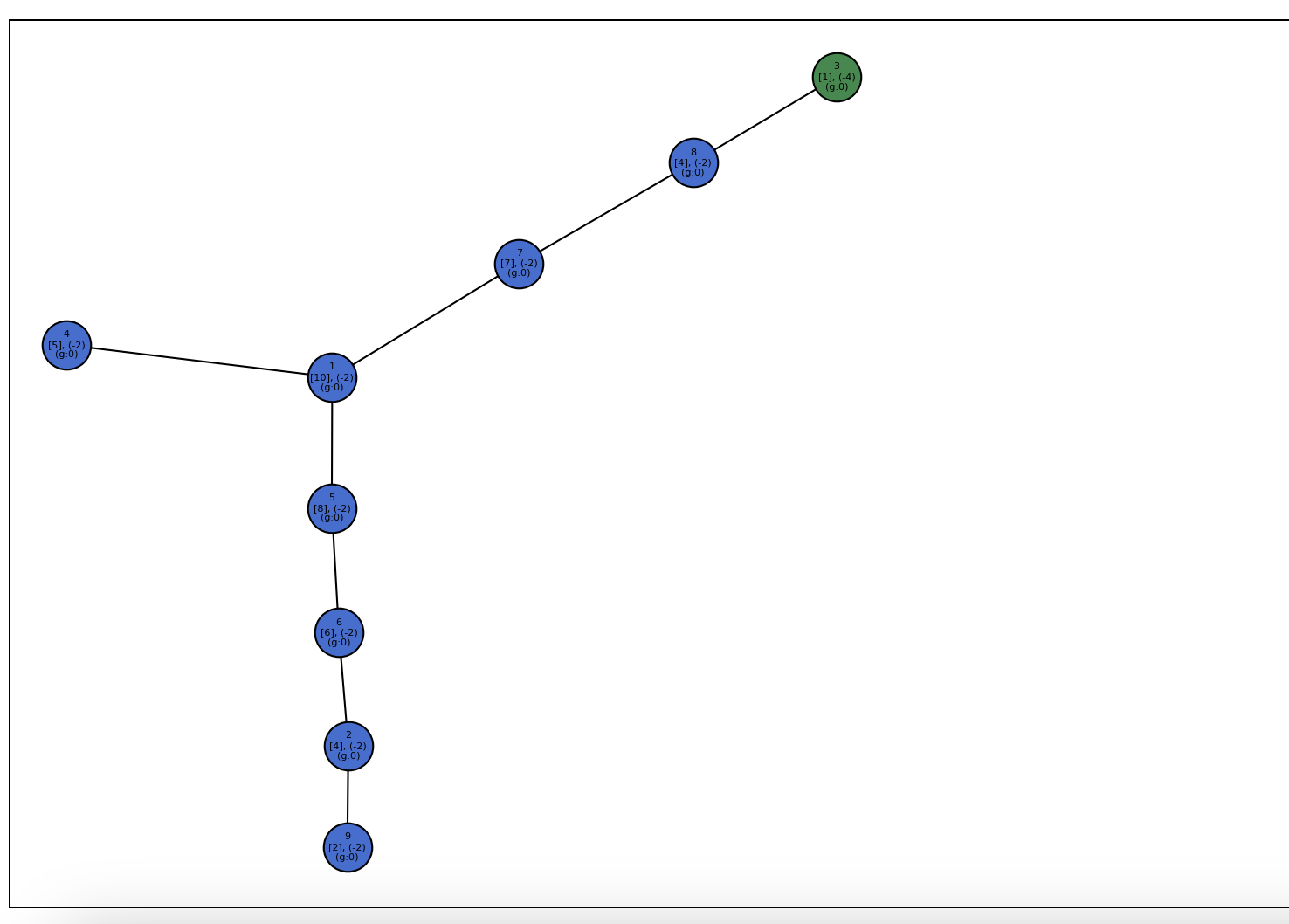}
    \caption{The minimal resolution graph of $\phi_1^3: t^3 = -x^5+y^2$.}
    \label{fig:phi13}
    \end{center}
\end{figure}
\end{example}

\begin{example}\label{ex_psi}
Similarly, we obtain the minimal resolution graph of the singularity $t^6 = x(x^3-y^2)$, denoted by $\psi_1^6$ in \cite{SV2}, using the programs \verb|PrintGraphwithArrows.m| and then \verb|PlotGraph (Sage)|. We show the resulting graph in Figure \ref{psi1_6_wa}. Note that this graph is the graph given in the third row, first column of Figure 5 in \cite{SV2}.
\begin{figure}[ht]
 \includegraphics[width=4.5in]{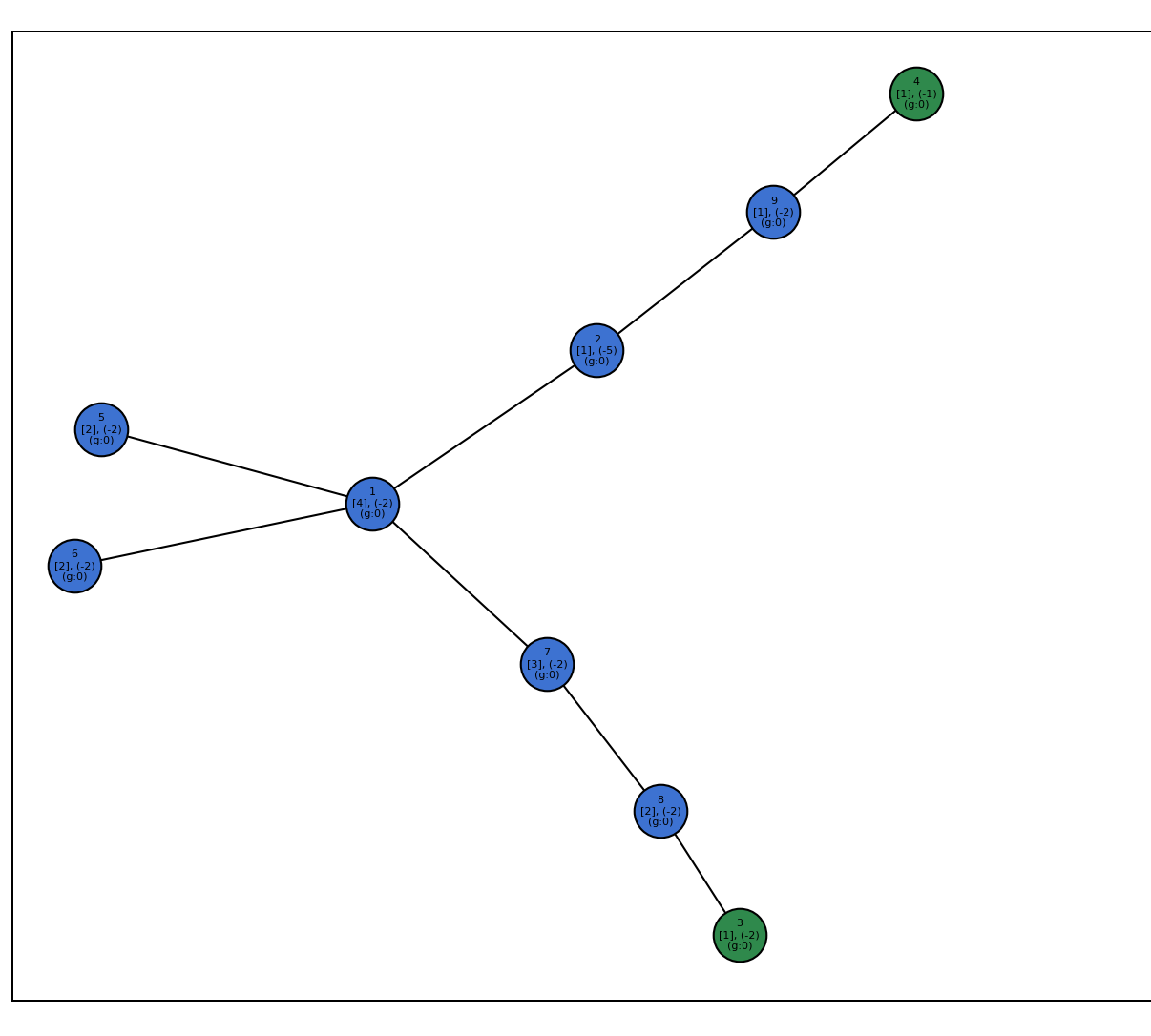}
    \caption{The minimal resolution graph of $\psi_1^6: t^6 = x(x^3-y^2)$.}
    \label{psi1_6_wa}
\end{figure}
\end{example}

\begin{remark}
In resolving the family \(\phi_2^k: y^2=x(x^4+t^k)\), \(k=1,\dots,8\), in \cite{SV}, we first resolved the plane curve singularities \(x(x^4+t^k)\), for all \(k=1,\dots,8\), on the \(xt\)-plane by iterated blow-ups. We then applied Némethi's algorithm, with the exponent of \(y\) fixed to \(2\), to the resulting resolution graphs in order to resolve the singularities \(y^2=x(x^4+t^k)\) in the \(xyt\)-space.

Therefore, to use \verb|PrintGraphwithArrows.m| for this family, one renames the variables and writes the family in the form \(\phi_2^m: t^2=x(x^4+y^m)\), where \(m\) is a positive integer. (The program computes the resolution graphs for all positive integers $m$, not only for the cases \(m=1,\dots,8\) considered in \cite{SV}). To compute the resolution graphs of the members of the \(\phi_2^m\) family, one fixes the program parameter \verb|k| corresponding to the exponent of \(t\) to be \(2\) and varies the exponent of \(y\) in the input section of \verb|PrintGraphwithArrows.m|. For example, to compute the resolution graph of \(\phi_2^8\), one uses the following input:
{
\normalfont
\begin{lstlisting} 
//Input the value for the power of t: 
k := 2;
//Input a polynomial in terms of x and y (and k if needed): 
curve := x*(x^4+y^8); 
\end{lstlisting}}
\begin{figure}
    \centering
    \includegraphics[width=0.75\linewidth]{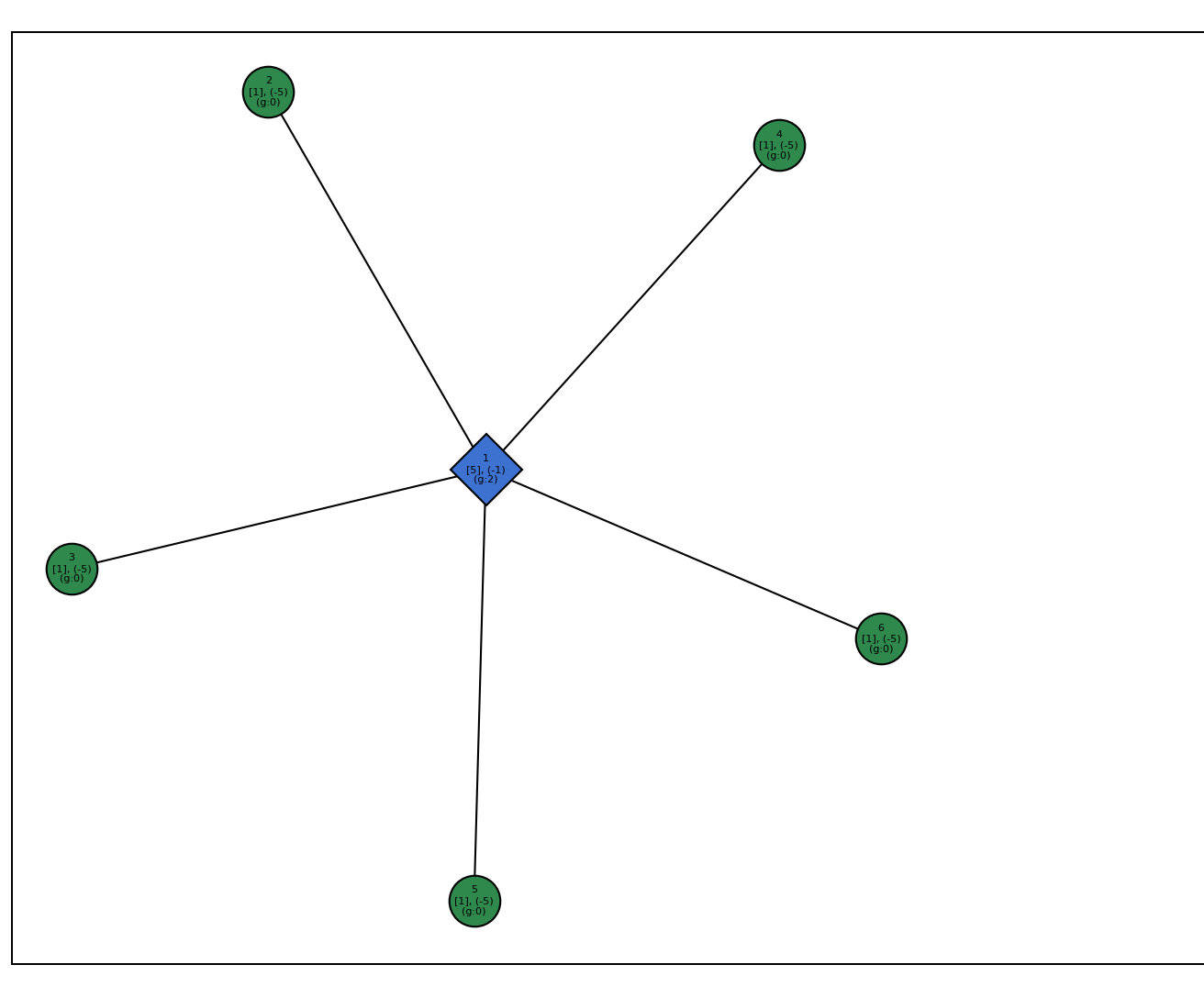}
    \caption{The minimal resolution graph of \(\phi_2^8: t^2=x(x^4+y^8)\).}
    \label{phi28}
\end{figure}
The resulting graph, shown in Figure \ref{phi28}, is minimal and isomorphic to the graph in the second column of the last row of Figure 8 in \cite{SV} after blowing down its $(-1)$ sphere. We compare it with the graph in the second column because the third column shows the resolution graphs after the arrows have been dropped, whereas our computation here retains all the arrows. The same procedure applies for any positive integer \(m\) for the members of the \(\phi_2^m\) family. 

An analogous remark applies to the family \(\phi_4^k\) studied in \cite{SV}.
\end{remark}

\begin{example} {
\normalfont \textbf{(Example 1.20(e) in \cite{Nemethi-lectures})}} We consider the singularity \(x^3+xy^3+t^2\), whose minimal resolution graph, after dropping the arrows, is the graph \(-E_7\) given in \cite[Example 1.20(e), p.9]{Nemethi-lectures}. We compute its resolution graph with arrows using our program \verb|PrintGraphwithArrows.m| and visualize the output with \verb|PlotGraph (Sage)|. The resulting graph is shown in Figure~\ref{Nem1}. After removing the two arrows from this graph, we obtain precisely the graph \(-E_7\).
\begin{figure}
    \includegraphics[width=0.7\linewidth]{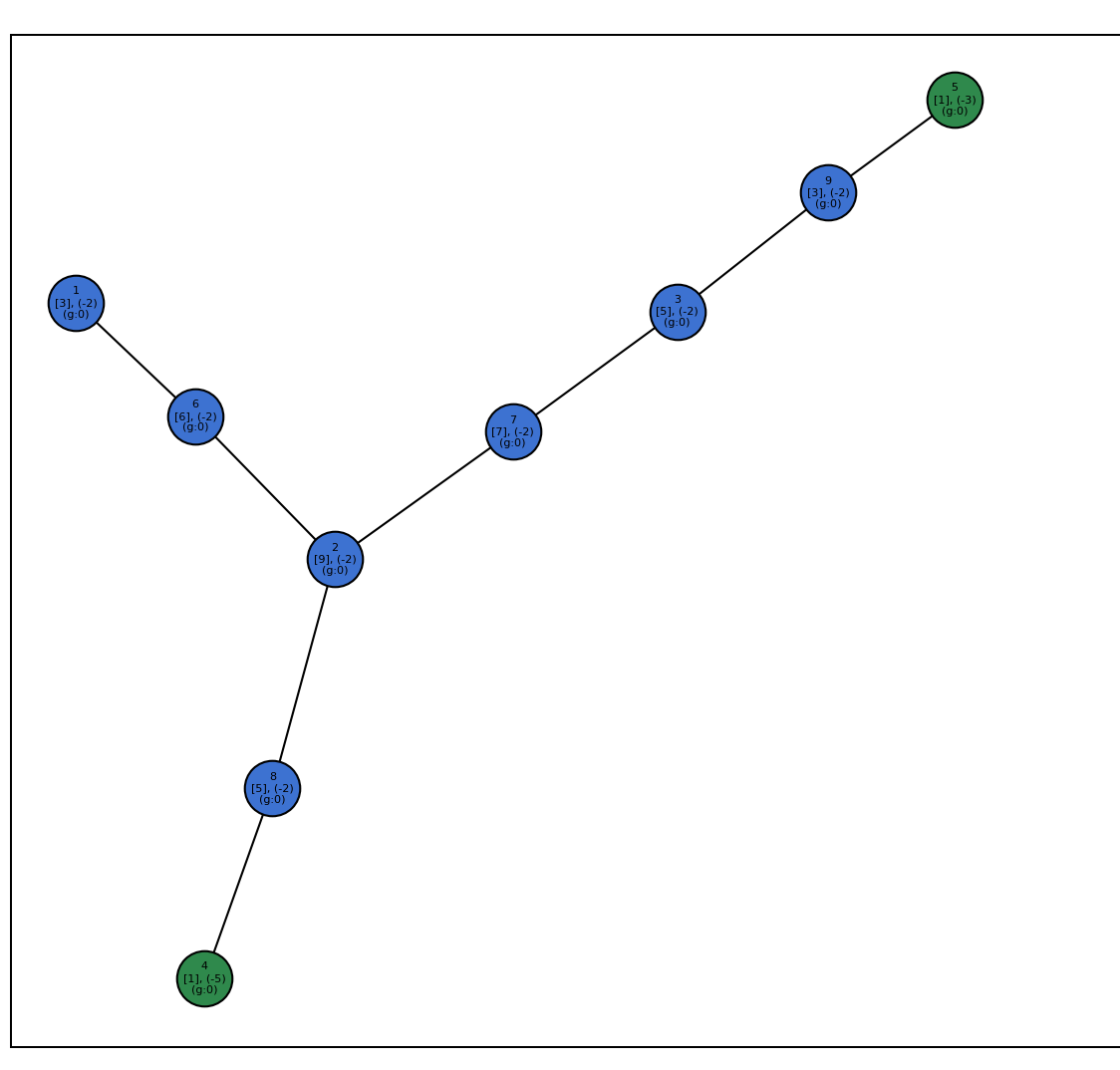}
    \caption{The minimal resolution graph of $x^3+xy^3+t^2=0$. The singularity is taken from Example 1.20(e) in \cite{Nemethi-lectures}.}
    \label{Nem1}
\end{figure}
\end{example}

\begin{example} {
\normalfont \textbf{(Example 1.22(d) in \cite{Nemethi-lectures})}} 
Next, we consider the singularity \((x^2+y^3)(x^3+y^2)+t^2\). Its minimal resolution graph, after dropping the arrows, is the multigraph given in \cite[Example~1.22(d), p.10]{Nemethi-lectures}. This multigraph consists of two vertices, each with self-intersection \((-3)\), joined by two parallel edges. Applying our programs to \((x^2+y^3)(x^3+y^2)+t^2\) produces the graph shown in Figure~\ref{multi}. After dropping the two arrows and blowing down the two \((-1)\)-spheres, we recover the same multigraph.
\begin{figure}
    \centering
    \includegraphics[width=0.7\linewidth]{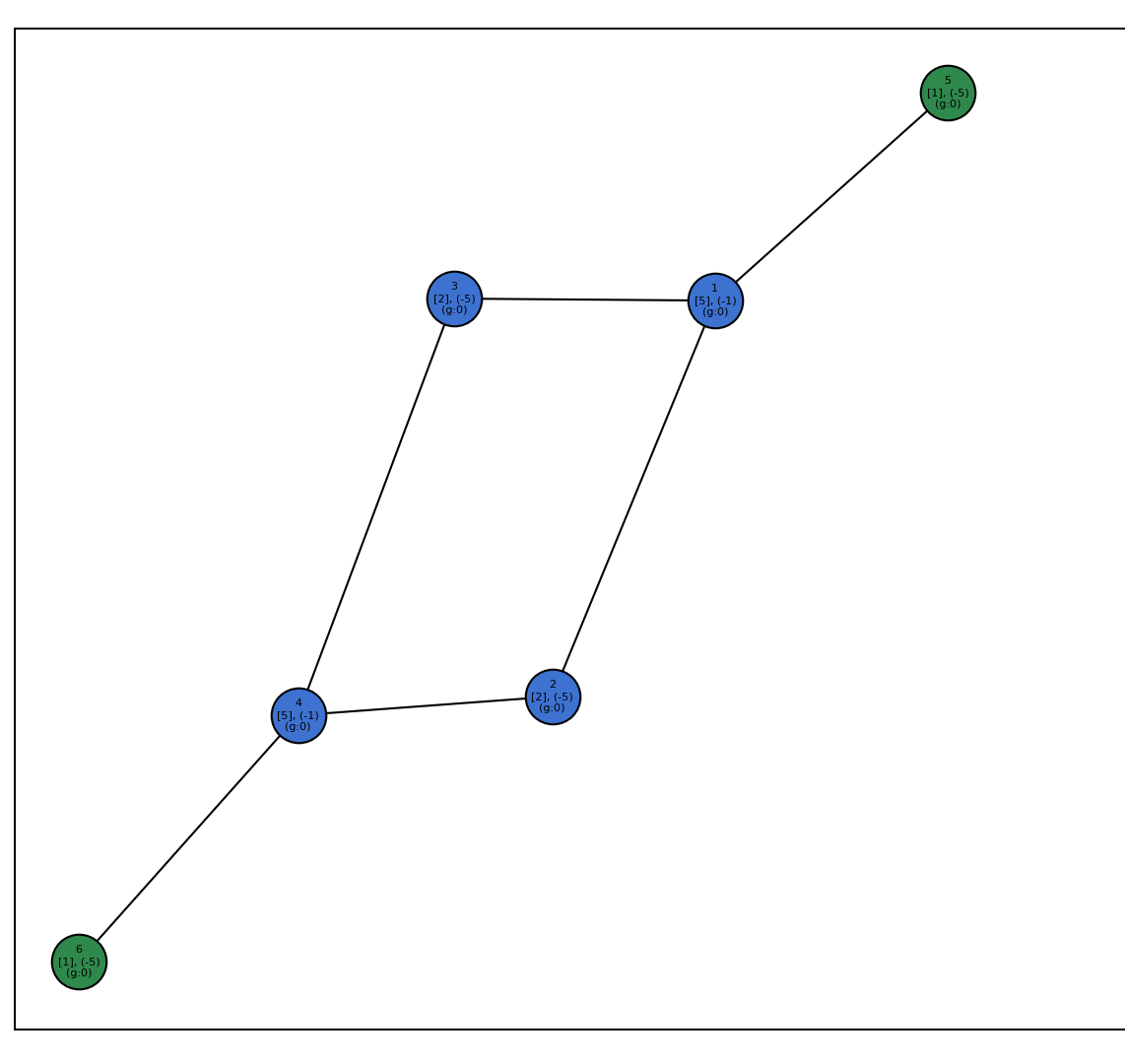}
    \caption{Resolution graph of $(x^2+y^3)(x^3+y^2)+t^2=0$. The singularity is taken from Example 1.22(d) in \cite{Nemethi-lectures}.}
    \label{multi}
\end{figure}
\end{example}

\begin{example}{
\normalfont \textbf{(Exercise 7.2.15(b) in \cite{GS})}} The singularity $t^2 = x^3+xy^5$ is given in Exercise 7.2.15(b) in \cite{GS} and its minimal resolution graph is given in Figure 12.60(b), p.506, in \cite{GS}. The first graph in Figure 12.60(b) shows the resolution graph of the plane curve singularity $B= x^3+xy^5=0$, where $\widetilde B$'s denote the proper transforms of the irreducible components of $B$. The second graph shows the non-minimal resolution graph of the suspension singularity $t^2 = x^3+xy^5$ before the blowdowns, with the arrows (the preimages of the $\widetilde B$ curves) are dropped. For the reader's convenience, we reproduce this second graph in Figure \ref{GSS}, together with the successive blowdown steps leading to the third graph of Figure 12.60(b) in \cite{GS}. This third graph is the minimal resolution graph of $t^2=x^3+xy^5$. 

Now we apply our programs to the singularity $t^2 = x^3+xy^5$. We show the resulting graph in Figure \ref{Our}. We note that when we drop the two arrows in Figure \ref{Our}, we obtain the second graph in Figure \ref{GSS}. Thus, after the blowdown process, we again obtain the same minimal resolution graph.

\begin{figure}
    \centering
    \includegraphics[width=1\linewidth]{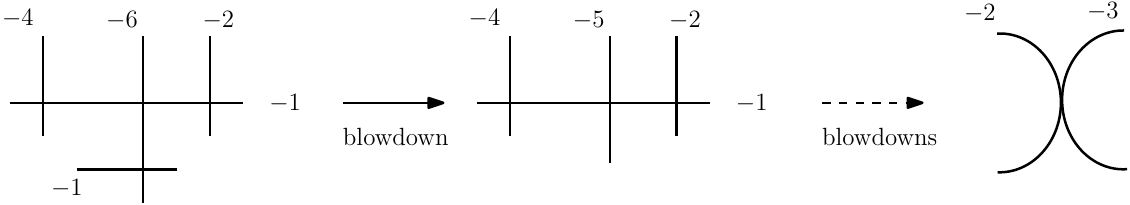}
    \caption{The minimal resolution graph of the suspension singularity $t^2 = x^3+xy^5$ from \cite{GS}, Exercise 7.2.15(b).}
    \label{GSS}
\end{figure}
\begin{figure}
    \centering
    \includegraphics[width=0.7\linewidth]{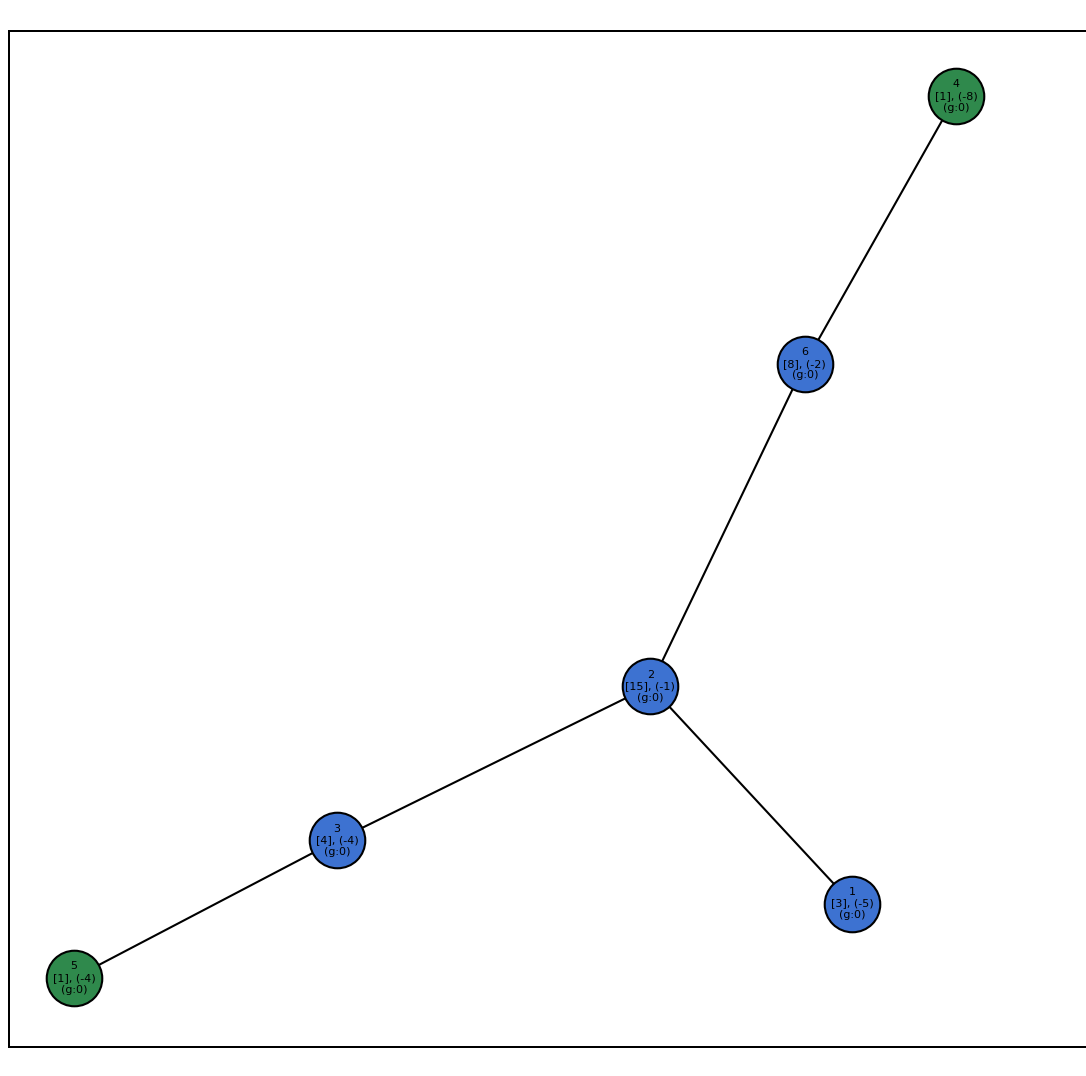}
    \caption{Resolution graph of the suspension singularity $t^2 = x^3+xy^5$. This is the same as the second graph of Figure \ref{GSS}. The singularity is taken from \cite{GS}, Exercise 7.2.15(b).}
    \label{Our}
\end{figure}
\end{example}

\begin{remark}
\label{exx}
The examples above are representative of the computations performed by the programs \verb|PrintGraphwithArrows.m| and \verb|PlotGraph (Sage)|. Beyond these examples, we have applied our programs to all 37 members of the five families of suspension singularities studied in \cite{SV} and \cite{SV2} (listed in \eqref{fam}), obtaining the expected resolution graphs in each case. We have also applied the programs to other parts of Examples 1.20 and 1.22 in \cite{Nemethi-lectures}, as well as to Example 7.2.11 and all parts of Exercises 7.2.12 and 7.2.15 in \cite[Section~7.2]{GS}, recovering the corresponding resolution graphs in each case. Moreover, our programs are applicable to polynomials of arbitrary degrees and are not limited to low-degree cases such as those listed in \eqref{fam}.
\end{remark}

\section{Periodicity and Isomorphisms of Minimal Resolution Graphs}
\label{Theorem}
In this section, we prove \autoref{T} on compactified minimal resolution graphs of suspension singularities. Compactified minimal resolution graphs are the graphs obtained by taking the one-point compactification of the curves denoted by arrows in the graphs. In \cite{SV,SV2}, we showed that compactified minimal resolution graphs of the following suspension singularities: 
\begin{equation} \label{finiteorder}
\begin{aligned}
\phi_1^k &: y^2 = x^5+t^k, \quad k=1,\dots, 9, \\
\phi_3^k &: y^2 = x^6+t^k, \quad k=1,\dots, 5,  \\
\psi_1^k &: t^k = x(x^3-y^2), \quad k=1,\dots, 7
\end{aligned}
\end{equation}
are precisely the fibers with finite-order monodromies in the Namikawa–Ueno classification, and they exhaust all genus two singular fibers with finite-order monodromies. We also showed that compactified minimal resolution graphs of \(\phi_1^{10}\), \(\phi_3^6\), and \(\psi_1^8\) are smooth, closed, genus two surfaces, namely the generic fibers of the corresponding closed, genus two singular fibrations. In addition, the monodromy factorizations associated with the \emph{compactified} minimal resolution graphs of $\phi_1$, $\phi_3$, and $\psi_1$ are
\[
\phi_1=\tau_1\tau_2\tau_3\tau_4,\qquad
\phi_3=\tau_1\tau_2\tau_3\tau_4\tau_5,\qquad
\psi_1=\tau_4\tau_3\tau_2\tau_1\tau_{1},
\]
which are well known to have orders \(10\), \(6\), and \(8\), respectively, in the mapping class group of the \emph{closed} genus two surface. Building on these results, we now prove that for each of the following infinite families
\begin{equation} 
\begin{aligned}
\phi_1^k &: y^2 = x^5+t^k, \quad k \in \mathbb{Z}_{>0}, \\
\phi_3^k &: y^2 = x^6+t^k, \quad k \in \mathbb{Z}_{>0},  \\
\psi_1^k &: t^k = x(x^3-y^2), \quad k \in \mathbb{Z}_{>0},
\end{aligned}
\end{equation}
the corresponding compactified minimal resolution graphs are periodic, with periods equal to the orders of the corresponding monodromy factorizations, namely \(10\), \(6\), and \(8\), respectively:
\begin{theorem}
     \label{T} 
For all positive integers $p,q$, the following statements hold:
\begin{itemize}
    \item if $p \equiv q \pmod{10}$ then the compactified minimal resolution graphs of $\phi_1^p: y^2 = x^5+t^p$ and $\phi_1^q: y^2 = x^5+t^q$ are isomorphic,
     \item if $p \equiv q \pmod{6}$ then the compactified minimal resolution graphs of $\phi_3^p: y^2 = x^6+t^p$ and $\phi_3^q: y^2 = x^6+t^q$ are isomorphic,
      \item if $p \equiv q \pmod{8}$ then the compactified minimal resolution graphs of $\psi_1^p: t^p = x(x^3-y^2)$ and $\psi_1^q: t^q = x(x^3-y^2)$ are isomorphic,
\end{itemize}
where the isomorphisms preserve all graph decorations; that is, the corresponding components of two isomorphic graphs have the same self-intersections, genera, and multiplicities.
\end{theorem}
The converses of the statements in Theorem \ref{T} also hold except that, after compactifying the arrows, the pair
\[
(\phi_3^2,\ \phi_3^4)
\]
has isomorphic minimal resolution graphs (see Figure 6 in \cite{SV}). Consequently, the same holds whenever
\[
p \equiv 2 \pmod{6}
\quad\text{and}\quad
q \equiv 4 \pmod{6}.
\]
Likewise, each of the pairs
\[
(\psi_1^1,\psi_1^3),\qquad
(\psi_1^2,\psi_1^6),\qquad
(\psi_1^5,\psi_1^7)
\]
has isomorphic minimal resolution graphs after compactifying the arrows (see Figures 1, 2, 3, 4 in \cite{SV2}). Consequently, the corresponding congruence classes modulo \(8\) also yield isomorphic minimal resolution graphs.

We note that for the \(\phi_2\) and \(\phi_4\) families considered in \cite{SV}, we do not have analogous periodicity results, since applying Némethi's algorithm to these families produces only subsets of the corresponding Namikawa--Ueno fibers rather than the complete fibers. The complete fibers associated with the \(\phi_2\) and \(\phi_4\) families were obtained in \cite{SV2}; these are precisely the fibers arising from the \(\psi_1\) family.

\begin{example}
\label{psi_78}
To illustrate the periodicity of the minimal resolution graphs stated in Theorem \ref{T}, we compute the minimal resolution graph of $\psi_1^{78}$ using
\verb|PrintGraphwithArrows.m| and \verb|PlotGraph (Sage)|. The resulting graph is shown in Figure \ref{psi78}. By compactifying the two arrows in Figure \ref{psi78} and blowing them down, we obtain a graph isomorphic to the minimal resolution graph of $\psi_1^6$ given in Example \ref{ex_psi}, after compactifying its arrows and blowing down the resulting $(-1)$-vertex.
\begin{figure}[!ht]
    \centering
\includegraphics[width=0.8\linewidth]{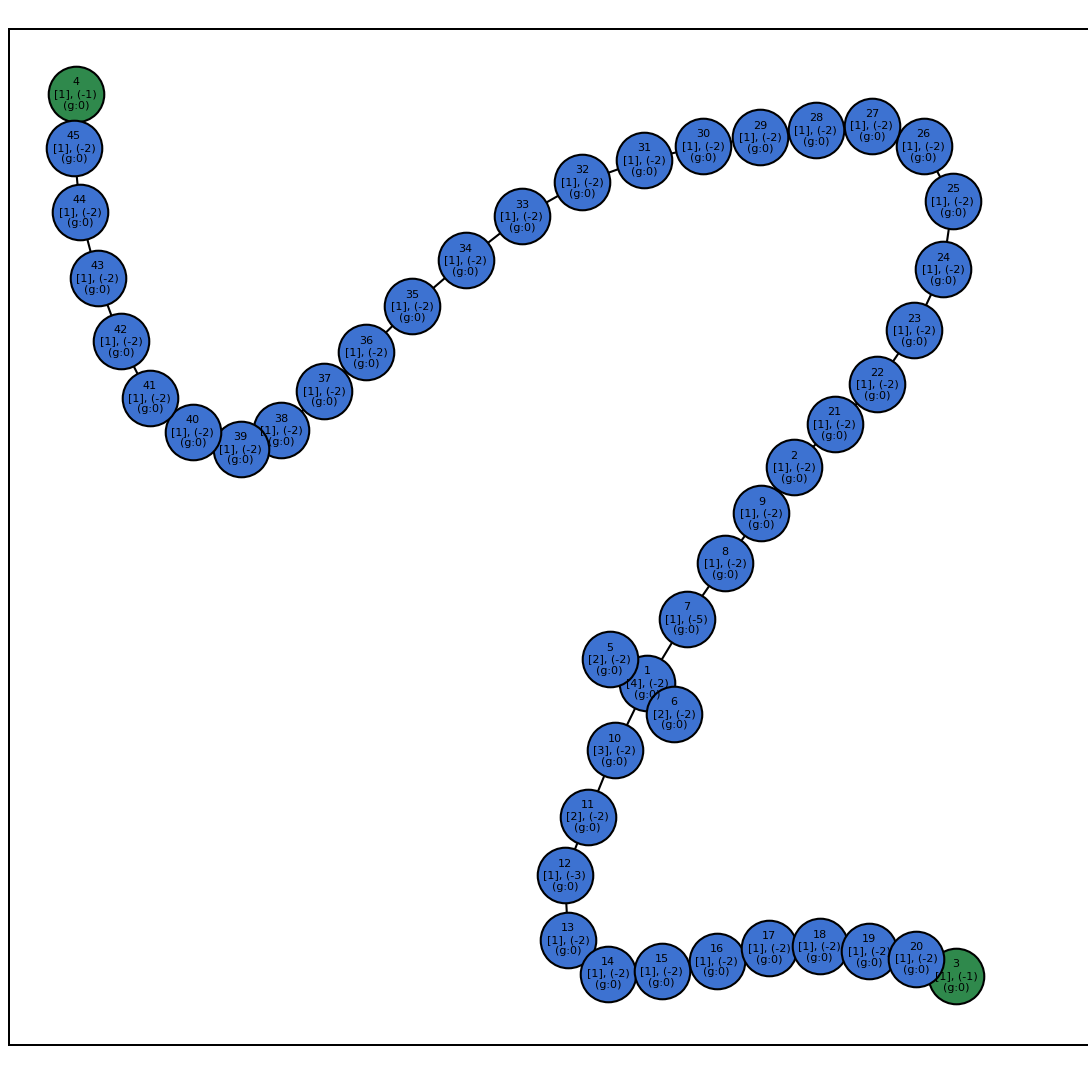}
    \caption{The minimal resolution graph of $\psi_1^{78}$.}
    \label{psi78}
\end{figure}

For larger exponents, the graphs produced by \verb|PlotGraph (Sage)| are not always clearly displayed in the online SageMath environment, which may render them as if they contained loops, even though the underlying graphs have no loops. One needs to run SageMath cell several times to see the correct display. If the graph is too large or too small, one can modify line 31, parameter \verb|figsize=(10, 10)| to make a bigger or smaller display. This allows for the size of the graph to be accommodated.
\end{example}
\begin{proof}
Theorem \ref{T} is proved by a direct application of Némethi's algorithm. We give the proofs for three representative congruence classes, $[\phi_3^2], [\phi_1^1]$, and $[\phi_1^{10}]$.

We start with the congruence class $[\phi_3^2: y^2 = x^6+t^2]$ modulo 6. Using the programs \verb|PrintGraphwithArrows.m| and \verb|PlotGraph (Sage)|, we determine the minimal resolution graph of the singularity $\phi_3^2$, depicted on the left of Figure \ref{phi32}. When we compactify the arrows we obtain the graph on the right.
\begin{figure}
    \centering
    \includegraphics[width=0.7\linewidth]{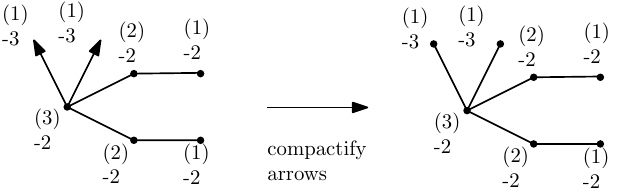}
    \caption{The minimal resolution graph of $\phi_3^2$.}
    \label{phi32}
\end{figure}

Then, we manually compute the resolution graphs of $\phi_3^n$ for all $n \equiv 2$ (mod 6), and $n>2$, and show that the resulting compactified graphs are the same as the graph in Figure \ref{phi32}. We first resolve the plane curve singularity $-x^6+y^2$ by successive blow-ups. We obtain the graph shown in Figure~\ref{62}, where we denote its vertices by $w_1$, $w_2$, and $w_3$. The multiplicities of the components are given in parentheses.
\begin{figure}
    \centering
    \includegraphics[width=0.26\linewidth]{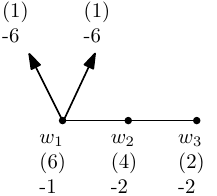}
    \caption{Resolution graph of the plane curve singularity $-x^6+y^2=0$.}
    \label{62}
\end{figure}
Then, we apply Némethi's algorithm to the graph in Figure \ref{62} taking the special degree-$n$ cover prescribed by the algorithm. That is to say, we determine the preimages of all vertices, edges, and arrows of the graph in Figure \ref{62} under this cover. 

From part (a) of the algorithm on p.14 of \cite{Nemethi-lectures}, we compute that there is one vertex lying over the vertex $w_1$ and two vertices lying over each of the vertices $w_2$ and $w_3$. We then find that all the resulting vertices have genus $0$. We note that in these computations for each vertex $w_i$, we use the multiplicities of its neighbors.

From part (c) of the algorithm, we have that above each arrow there is one string of type $G(6,1,n)$. Using the formula for strings on p.13 of \cite{Nemethi-lectures}, we determine this string. Namely, we solve
\[1+x\frac{6}{\gcd(6,n)} \equiv 0 \pmod{\frac{n}{\gcd(6,n)}}.\]
Writing $n=6p+2$, where $p\geq 1$, this becomes
\[1+3x \equiv 0 \pmod{3p+1}.\]
The least nonnegative solution is $x=p$. Then we compute the Hirzebruch--Jung continued fraction expansion of $\displaystyle\frac{3p+1}{p}$ which is 
\begin{equation} 
\label{arrows}
[-4, \underbrace{-2,\dots,-2}_{(p-1)\ }].
\end{equation}
Hence, above each arrow we obtain a string starting at the preimage of $w_1$ and ending with an arrowhead, and the internal vertices have self intersections $[-4, -2, \dots, -2]$. From formula $(**)$ on p.13 of \cite{Nemethi-lectures} we find the multiplicity $m_1$ of the first vertex in the string that has self intersection $(-4)$. Namely, we have $1+3p = m_1 (3p+1)$ and hence $m_1= 1$. Then, using the first two equations on p.14 of \cite{Nemethi-lectures}, which are obtained from formula $(*)$ on p.13, we find that all multiplicities of the internal vertices of the string are equal to $1$.  

Now we find the preimage of the edge between the vertices $w_1$ and $w_2$ of the graph in Figure \ref{62}. From part (b) of the algorithm on p.14 we have that there are $\gcd(6,4,n)=2$ strings of type $G(3,2,n/2) = G(3,2,3p+1)$. To determine these strings, we solve 
\[2+x\frac{3}{\gcd(3,3p+1)} \equiv 0 \pmod{\frac{3p+1}{\gcd(3,3p+1)}}.\]
The least nonnegative solution is $x=2p$. Hence we take the Hirzebruch--Jung continued fraction expansion of $\displaystyle\frac{3p+1}{2p}$ which is 
\begin{equation}
\label{12}
\begin{cases}
  [-2], & \text{if } p=1 \\
  [-2,-3,\underbrace{-2,\dots,-2}_{(p-3)/2\ }], & \text{if } p \geq 3\;\; \text{and odd}\\
  [-2, -4], & \text{if } p=2 \\
  [-2,-3,\underbrace{-2,\dots,-2,}_{(p-4)/2\  } -3], & \text{if } p \geq 4\;\; \text{and even.}\\
\end{cases}
\end{equation}
As in \eqref{arrows}, the sequences in \eqref{12} give the self-intersections of the internal vertices of the corresponding strings in each case. For each $p$, there are two identical strings, both starting at the preimage of the vertex $w_1$, with one ending at the preimage $w_{2,1}$ of the vertex $w_2$ and the other ending at the other preimage $w_{2,2}$ of $w_2$.

In the same way, we find the preimage of the edge between the vertices $w_2$ and $w_3$ of the graph in Figure \ref{62}. There are $\gcd(4,2,n)=2$ strings of type $G(2,1,3p+1)$. We solve
\[1+x\frac{2}{\gcd(2,3p+1)} \equiv 0 \pmod{\frac{3p+1}{\gcd(2,3p+1)}}.\]
The least nonnegative solution is 
\begin{equation*}
x= \begin{cases}
  \displaystyle\frac{3p-1}{2} \pmod{\frac{3p+1}{2}}, & \text{if $p$ is odd} \\
  \displaystyle\frac{3p}{2} \pmod{3p+1}, & \text{if $p$ is even} 
\end{cases}
\end{equation*}
Hence we take the following Hirzebruch--Jung continued fraction expansions:
\begin{equation}
\label{23}
\begin{cases}
  \displaystyle\frac{3p+1}{3p-1} = [\underbrace{-2,\dots,-2}_{(3p-1)/2\ }], & \text{if $p$ is odd } \\
\displaystyle\frac{2(3p+1)}{3p} = [-3, \underbrace{-2,\dots,-2}_{(3p/2)-1\ }], & \text{if $p$ is even. } 
\end{cases}
\end{equation}
The sequences in \eqref{23} similarly give the self-intersections of the internal vertices of the corresponding strings. For each $p$, there are two identical strings. One starts at the preimage $w_{2,1}$ of the vertex $w_2$ and ends at the preimage $w_{3,1}$ of the vertex $w_3$. The other starts at the preimage $w_{2,2}$ of the vertex $w_2$ and ends at the preimage $w_{3,2}$ of the vertex $w_3$.

As we did for the string corresponding to \eqref{arrows}, we also compute the multiplicities of all vertices in the strings corresponding to \eqref{12} and \eqref{23}.

Then, for each of the four cases $p=1$, $p=2$, $p \geq 3$ odd, and $p \geq 4$ even, we combine the strings corresponding to that case from \eqref{arrows}, \eqref{12}, and \eqref{23}. Next, using formula $(*)$ on p.13, we compute the missing self intersections, namely those of the preimages of the components of the graph in Figure \ref{62}. Hence we obtain the resolution graphs of $\phi_3^n$ for all $n=6p+2$ with $p \geq 1$ as shown in Figure \ref{cases}. The red vertices in Figure \ref{cases} represent the preimages of the vertices $w_i$ of the graph in Figure \ref{62}, and the black vertices represent the internal vertices of the strings.
\begin{figure}
    \centering
    \includegraphics[width=1\linewidth]{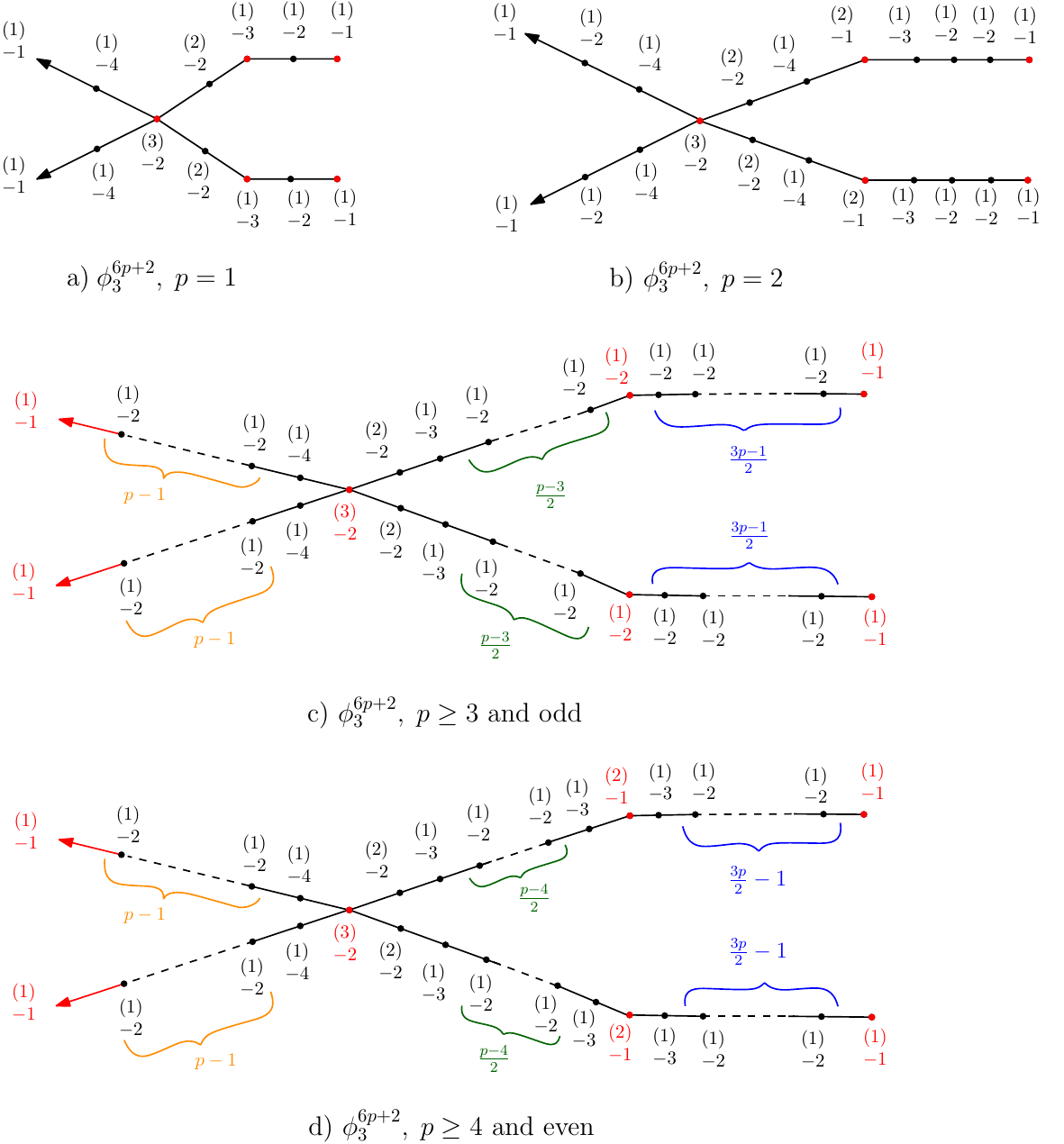}
    \caption{Resolution graphs of $\phi_3^n$ for all $n=6p+2$ with $p \geq 1$ before blowdowns.}
    \label{cases}
\end{figure}

After blowing down all $(-1)$ vertices in all the graphs shown in Figure \ref{cases}, we obtain minimal resolution graphs, all of which have the form shown on the left of Figure \ref{final}. This is the minimal resolution graph of $\phi_3^n$ for all $n\equiv 2\pmod{6}$, $n>2$. For any particular value of $n$ in this congruence class, this graph can be verified  using \verb|PrintGraphwithArrows.m| and \verb|PlotGraph (Sage)| without performing any of the above computations. We then compactify the arrows and perform a sequence of blowdowns to obtain the graph shown on the right of Figure \ref{final}. This is the same as the compactified graph of $\phi_3^2$ shown on the right of Figure \ref{phi32}. This completes the proof for the congruence class $[\phi_3^2]$ modulo 6.
\begin{figure}
    \centering
    \includegraphics[width=1\linewidth]{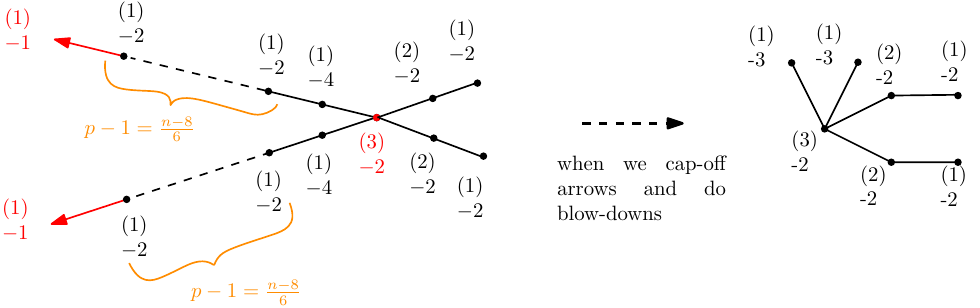}
    \caption{The minimal resolution graph of $\phi_3^n$ for all $n=6p+2$ with $p \geq 1$.}
    \label{final}
\end{figure}

In the proofs of the remaining congruence classes, we follow the same method. For the reader's convenience, we summarize the proofs for the congruence classes $[\phi_1^1]$ and $[\phi_1^{10}]$ modulo 10.

For $[\phi_1^1]$, we first compute the resolution graph of $\phi_1^1: y^2=x^5+t^1$ as shown in the first graph of Figure \ref{phi1_1}. This first graph can be verified by \verb|PrintGraphwithArrows.m| and \verb|PlotGraph (Sage)|. Then we manually blowdown the $(-1)$ spheres as shown in the figure. When we compactify the arrow at the end, we obtain the $(2,5)$-cusp.
\begin{figure}[!ht]
    \centering
    \includegraphics[width=1\linewidth]{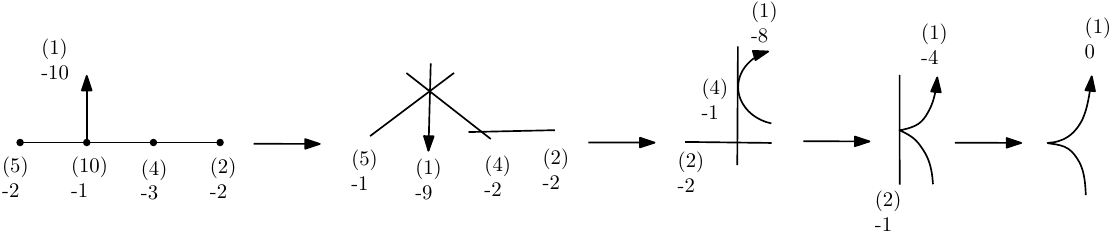}
    \caption{The minimal resolution graph of $\phi_1^{1}$.}
    \label{phi1_1}
\end{figure}

Next, we manually compute the resolution graph of $\phi_1^{11}:y^2=x^5+t^{11}$, shown in the first graph of Figure~\ref{phi1_11}. We then successively blow down the $(-1)$-spheres on the two arms of the graph, obtaining the second graph in Figure~\ref{phi1_11}. The second graph can be verified using \verb|PrintGraphwithArrows.m| and \verb|PlotGraph (Sage)| without computing the first graph. Finally, after compactifying the arrow and performing the remaining blowdowns, we again obtain the $(2,5)$-cusp, as shown in the subsequent graphs of Figure~\ref{phi1_11}.
\begin{figure}
    \centering
    \includegraphics[width=1\linewidth]{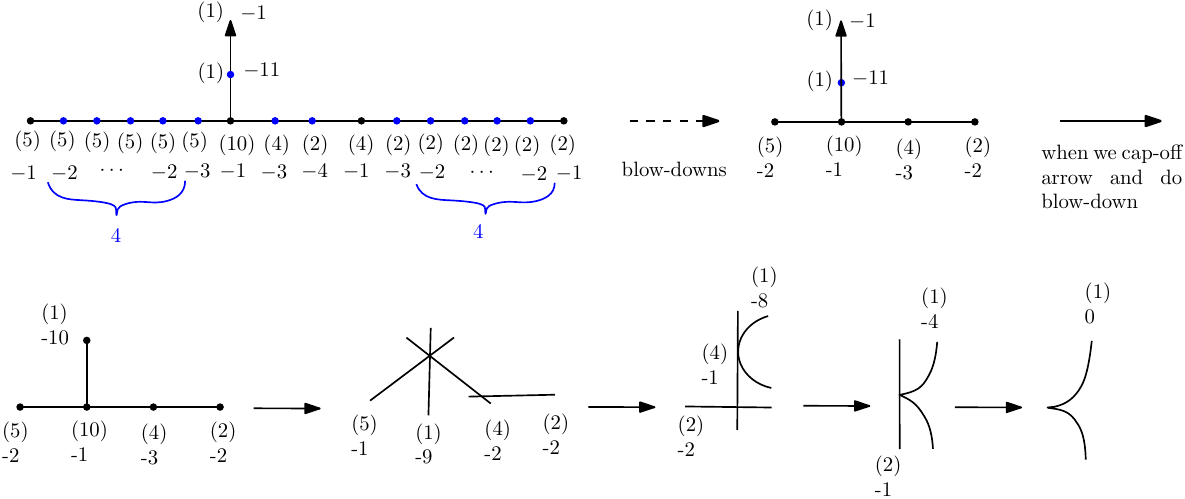}
    \caption{The minimal resolution graph of $\phi_1^{11}$.}
    \label{phi1_11}
\end{figure}

Then, we manually compute the resolution graph of $\phi_1^{10p+1}$, where $p\ge2$. The resulting graph is the first graph in Figure~\ref{phi1_10p+1}. We then successively blow down the $(-1)$-spheres and obtain the second graph in Figure~\ref{phi1_10p+1}. Finally, after compactifying the arrow and performing the remaining blowdowns, we again obtain the $(2,5)$-cusp. Hence, in all cases we obtain the $(2,5)$-cusp. This completes the proof for the congruence class $[\phi_1^1]$ modulo $10$.
\begin{figure}
    \centering
    \includegraphics[width=1\linewidth]{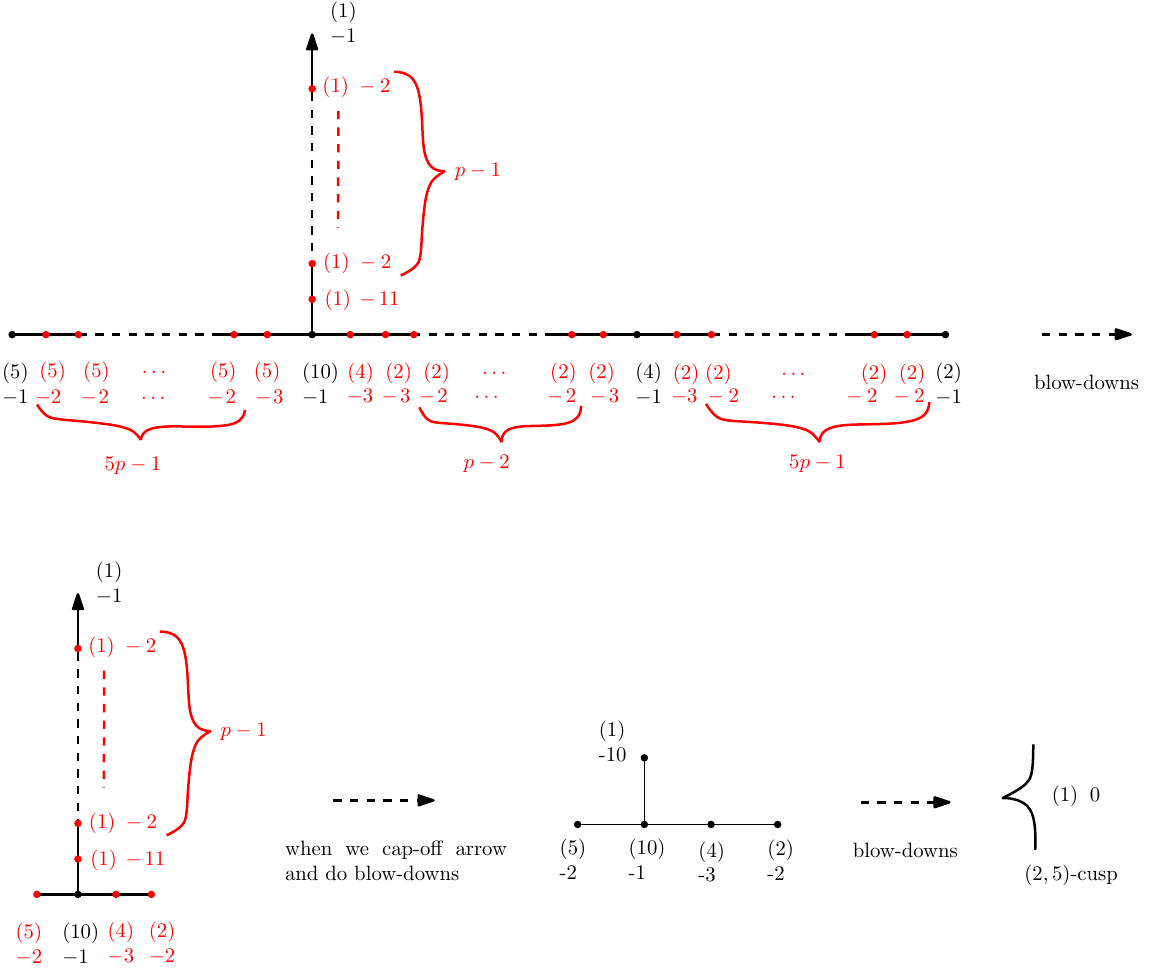}
    \caption{The minimal resolution graph of $\phi_1^{10p+1}$, $p \geq2$.}
    \label{phi1_10p+1}
\end{figure}

Now we give the proof for the congruence class $[\phi_1^{10}]$ modulo $10$. We first resolve $\phi_1^{10}$ as shown in Figure \ref{phi1_10}. After compactifying the arrow and blowing down all $(-1)$ spheres, we obtain the genus two surface of self intersection 0. 
\begin{figure}
    \centering
    \includegraphics[width=1\linewidth]{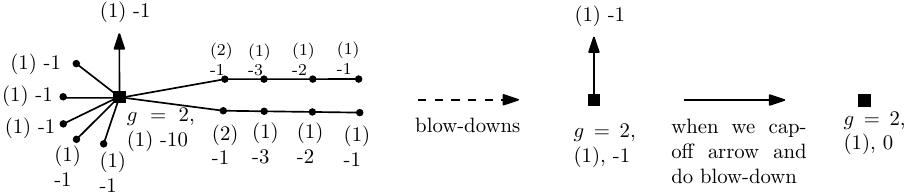}
    \caption{The minimal resolution graph of $\phi_1^{10}$.}
    \label{phi1_10}
\end{figure}

Next, we work with $\phi_1^{10p}$ with $p>1$. The resolution graph of $\phi_1^{10p}$ for odd $p>1$ is shown in part (a) of Figure~\ref{phi1_10p}, and the resolution graph for even $p$ is shown in part (b). In all of these graphs, the central vertices shown as squares represent genus two surfaces. After performing the blowdowns, all graphs simplify to the same graph, namely the third graph in Figure~\ref{phi1_10p}. Finally, after compactifying the arrow and performing the remaining blowdowns, we again obtain a genus two surface of self-intersection $0$, as in the case of $\phi_1^{10}$. This completes the proof for $[\phi_1^{10}]$ modulo $10$. 
The proofs for the remaining congruence classes are entirely analogous and are therefore omitted.
\begin{figure}
    \centering
    \includegraphics[width=1\linewidth]{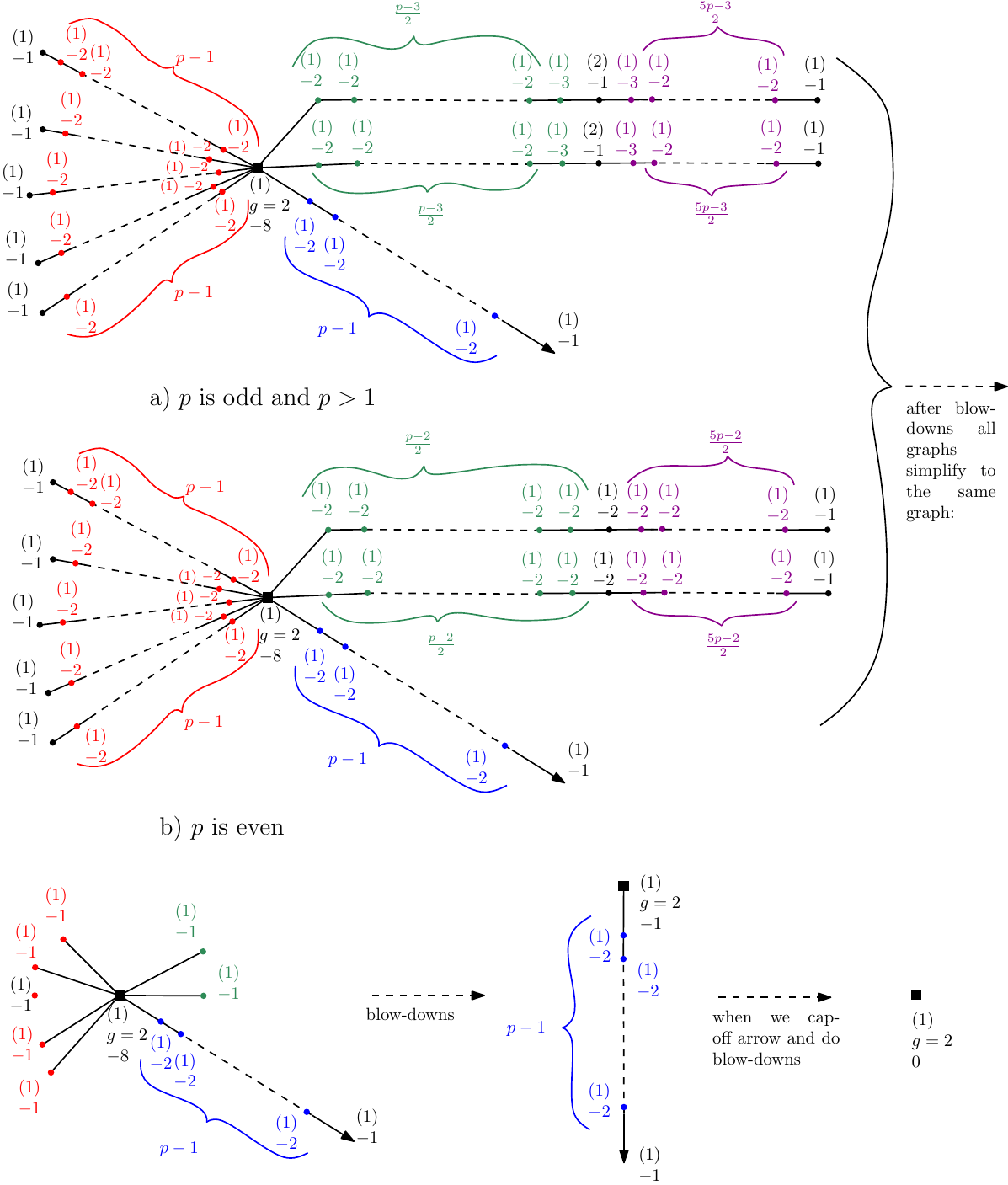}
    \caption{The minimal resolution graph of $\phi_1^{10p}$, $p>1$.}
    \label{phi1_10p}
\end{figure} 
\end{proof}

\section{Magma program for closed fibers}
\label{Magma2}
The Magma program \verb|PrintGraphwithArrows.m| distinguishes the arrows from the sphere components in the resulting graphs, and hence the arrows are not blown down. Throughout the remainder of the paper we work with closed fibers, obtained by compactifying the complex curves represented by the arrows in the resolution graphs. For this purpose, we provide a Magma program \verb|PrintGraphCompactified.m| in which the blowdown procedure treats the arrows as sphere components. Thus, they are blown down whenever their self-intersection is \((-1)\), just like any other \((-1)\)-sphere. 

The program \verb|PrintGraphCompactified.m| differs from \verb|PrintGraphwithArrows.m| only in this modification. Both programs are self-contained, containing all the necessary functions, and may therefore be used independently. We also note that Remark~\ref{bdown} applies equally to
\verb|PrintGraphCompactified.m|.

To produce the compactified minimal resolution graphs of suspension singularities, one downloads the Magma program \verb|PrintGraphCompactified.m| from \cite{SSG} and follows the same procedure described in \autoref{Magma1}, replacing the program \verb|PrintGraphwithArrows.m| with \verb|PrintGraphCompactified.m| before using \verb|PlotGraph (Sage)|.

\begin{example}
\label{ex_psi2}
We obtain the minimal resolution graph of $\psi_1^6:t^6 = x(x^3-y^2)$ by using \verb|PrintGraphCompactified.m| and then \verb|PlotGraph (Sage)|. We show the resulting graph in Figure \ref{psi1_6_c}. We note that this graph is isomorphic to the graph given in the third row and second column of Figure 4 in \cite{SV2}.
(Cf. Example \ref{ex_psi} above, where the arrows are retained.)
\begin{figure}
 \includegraphics[width=4in]{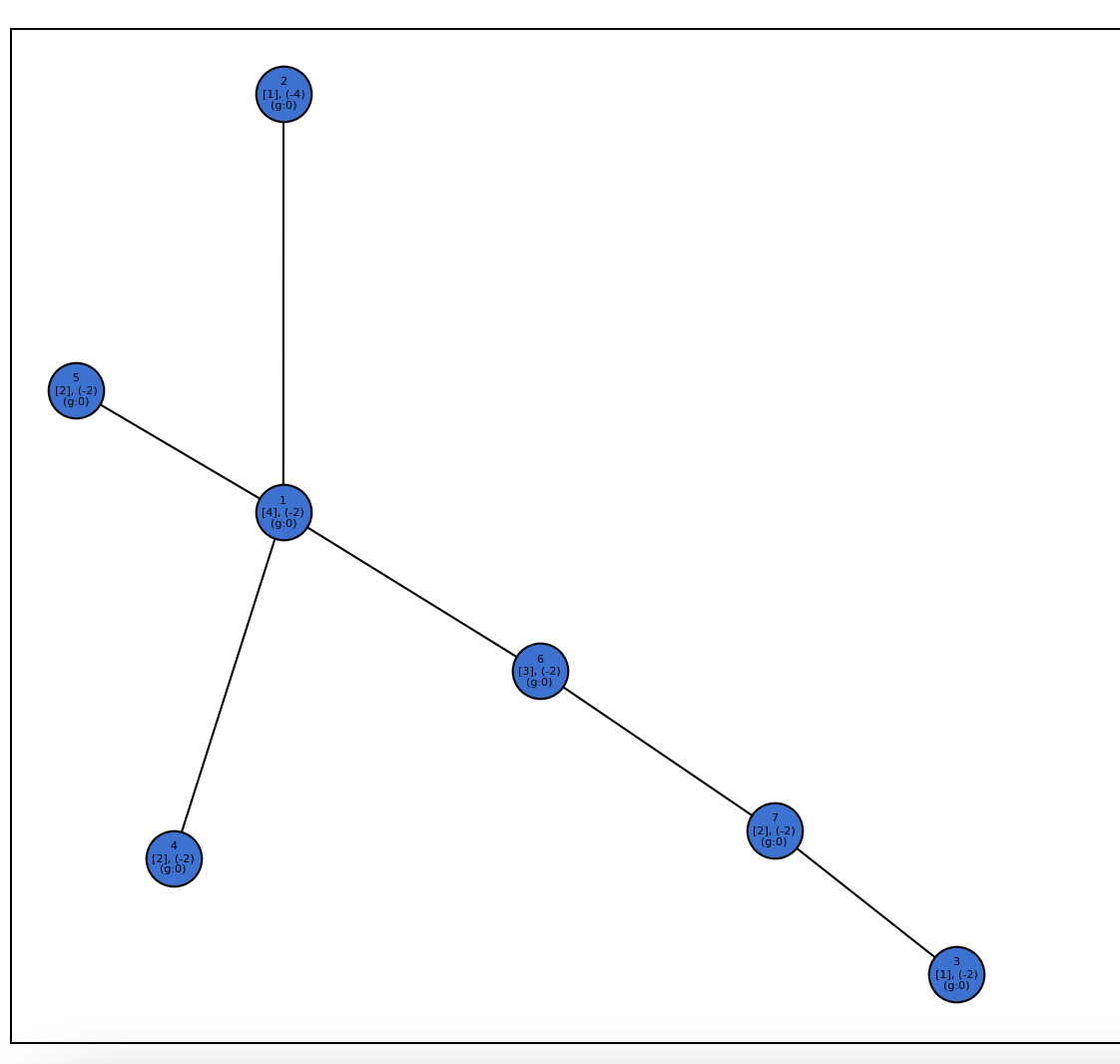}
    \caption{The minimal resolution graph of $\psi_1^6: t^6 = x(x^3-y^2)$ when arrows are compactified.}
    \label{psi1_6_c}
\end{figure}
\end{example}

As in \autoref{Magma1}, in addition to the example above, we have applied our programs \\
\noindent \verb|PrintGraphCompactified.m| and \verb|PlotGraph (Sage)| to all singularities considered in Remark \ref{exx} and recovered the expected resolution graphs in each case. We also note that, like \verb|PrintGraphwithArrows.m|, the program \verb|PrintGraphCompactified.m| is also applicable to polynomials of arbitrarily high degree. For example:

\begin{example}
    We compute the minimal resolution graph of $\psi_1^{78}: t^{78} = x(x^3-y^2)$ using \verb|PrintGraphCompactified.m| and then \verb|PlotGraph (Sage)|. We obtain a minimal graph which is precisely the minimal resolution graph of $\psi_1^6$ given in Example \ref{ex_psi2}, as expected. (Cf. Example \ref{psi_78} above, where the arrows are retained.)
\end{example}

\section{Magma Program for Isomorphisms of Closed Fibers}
\label{Isom}

In this section we provide a Magma program \verb|IsomorphismCheck.m|, available at \cite{SSG}, which could be of independent interest. It checks the isomorphism of resolution graphs of closed fibers over a specified parameter range. For a proposed period \(n\), the program first computes the resolution graphs of the closed fibers, as in \verb|PrintGraphCompactified.m|, for the parameters \(1,\dots,n\). Then it compares each subsequent graph with the graph corresponding to the same congruence class modulo \(n\). If all comparisons yield isomorphic graphs, the program confirms that, over the specified range of parameters, the graphs corresponding to the same congruence class modulo \(n\) are isomorphic.

To use the program, after downloading \verb|IsomorphismCheck.m| from \cite{SSG}, one specifies the desired input parameters. The parameter \texttt{finish} determines the upper bound of the range of exponents to be checked, namely $1 \leq k \leq \verb|finish|$, while \texttt{num} specifies the proposed period used in the graph comparisons. After specifying the input parameters, one loads the program in Magma:
\begin{lstlisting}
load "/path/to/IsomorphismCheck.m";
\end{lstlisting}

For example, the program \verb|IsomorphismCheck.m| can be applied to each of the three families in Theorem~\ref{T} by changing only the input parameters. For the family $\phi_1^k:y^2=x^5+t^k$, with range \(1\leq k \leq 1000\), one uses the following input:

\begin{lstlisting}
// ______________input________________

curve := x^5-y^2;

// k is the degree of t, for Nemethi's Algorithm 
// range of k
finish := 1000;

// period of the graphs
num    := 10;
\end{lstlisting}

For the family $\phi_3^k: y^2=x^6+t^k$, one replaces the input by:
\begin{lstlisting}
// ______________input________________

curve := x^6-y^2;

// k is the degree of t, for Nemethi's Algorithm 
// range of k
finish := 1000;

// period of the graphs
num    := 6;
\end{lstlisting}

Finally, for the family $\psi_1^k: t^k=x(x^3-y^2)$, one uses:
\begin{lstlisting}
// ______________input________________

curve := x*(x^3-y^2);

// k is the degree of t, for Nemethi's Algorithm 
// range of k
finish := 1000;

// period of the graphs
num    := 8;
\end{lstlisting}

For the $\phi_1$ family, \verb|IsomorphismCheck.m| returns:
\begin{lstlisting}
Isomorphisms HOLD for all k in [1.. 1000 ] with period  10
\end{lstlisting}
Analogous computations for the \(\phi_3\) and \(\psi_1\) families give the corresponding periodicity statements. For each of the three families \(\phi_1\), \(\phi_3\), and \(\psi_1\), \verb|IsomorphismCheck.m| performs all isomorphism checks for the range \(1\leq k \leq 1000\) in under \(10\) minutes on a standard machine. This range can be expanded at the expense of longer execution times.

\bibliographystyle{amsalpha}
\bibliography{References}

\vspace{0.2in}
\end{document}